\documentclass[12pt]{amsart}
\usepackage{graphicx}
\usepackage{amssymb}
\usepackage[T1]{fontenc}
\usepackage{geometry}
\usepackage{tikz}
\usepackage{float}
\usepackage{url}
\usepackage{mathtools}
\usepackage{amsmath,amsfonts,amsthm}
\usepackage{mathrsfs} 
\newcommand{\irr}{\operatorname{irr}}
\newcommand{\case}[1]{\paragraph*{Case #1:}}
\newtheorem{hypothesize}{Hypothesize}

\newtheorem{theorem}{Theorem}[section]
\newtheorem{lemma}[theorem]{Lemma}
\newtheorem{proposition}{Proposition}[section]

\newtheorem{definition}{Definition}

\author{Jasem Hamoud}
\address{\textbf{Jasem Hamoud} 
Department of Discrete Mathematics, Moscow Institute of Physics and Technology
}
\email{jasem1994hamoud@gmail.com}
\thanks{}
\author{Duaa Abdullah}
\address{\textbf{Duaa Abdullah:} Department of Discrete Mathematics, Moscow Institute of Physics and Technology }
\email{duaa1992abdullah@gmail.com}
\thanks{}
\title{Topological Indices in Trees  With Degree Sequence $\mathscr{D}$}
\date{}
\begin{document}
\begin{abstract}
In this paper, we denote a degree sequence as $\mathscr{D}=(d_1,d_2,\dots,d_n)$, ordered either non-increasingly ($d_1 \geqslant d_2 \geqslant \dots \geqslant d_n$) or non-decreasingly ($d_1 \leqslant d_2 \leqslant \dots \leqslant d_n$). The paper sets out to prove topological indices on trees with a degree sequence $\mathscr{D}$ which it is provides bounds that reflect extremal values for the number of edges incident to vertices, guided by the Albertson index (defined as $\sum_{uv \in E(G)} \lvert d_u(G) - d_v(G)\rvert$) and the Sigma index $\sigma(G)$ for trees $T$ with degree sequence $\mathscr{D}$.

For the first Zagreb index $M_1(T)$, we demonstrate that for a degree sequence $\mathscr{D}$ of order $n=4$, the irregularity measure of Albertson index satisfies:  
\[
\irr(T) = M_1(T)^2 - 2\sqrt{M_1(T)} + \sum_{i=1}^4 |x_i - x_{i+1}| - (d_2+d_3) - 1 .
\]
\end{abstract}
\maketitle
\noindent\rule{12.7cm}{1.0pt}

\medskip


\noindent\rule{12.7cm}{1.0pt}

\section{Introduction}\label{sec1}
Throughout this paper, let $G=(V,E)$ be a simple, connected graph, denote by $uv$ an edge connecting vertices $u$ and $v$. In 1995, Molloy and Reed in~\cite{Molloy} define a degree sequence as $\mathscr{D}=(d_1,d_2,\dots,d_n)$. Zhang et al. in~\cite{Zhang} define the classes of trees $\mathcal{T}$ among a non-increasing degree sequence $\mathscr{D}$, denote that by $\mathcal{T}_{\mathscr{D}}$ when $d_1\geqslant d_2\geqslant\dots \geqslant d_n$, in~\cite{Blanc} define $\mathcal{T}_{\mathscr{D}}$ as $\sum_{i=1}^{n}d_i=n-1$. The maximum degree $\Delta$ of any tree with a degree sequence $\mathscr{D}$ defined by Broutin et al. in~\cite{Broutin} as $\Delta_{\mathscr{D}}=\max \{i, \mathcal{T}_{\mathscr{D}}>0\}$.\par 
Further, let $\deg v=\deg(v)$ be the degree of a vertex $v\in V(G)$, the amount of diverge at $u$ equals $\deg u$, which for any tree has a centre chord of a maximum of one or two points. For a graph $G$ with $n$ vertices and $p$ pendents, Wai-Kai Chen in~\cite{Wai} established that the maximum number of edges is $\Delta(E) = \frac{1}{2}(n^2+n+p^2-3p)$. For each pendent $G_i$ (with $n_i$ vertices), the maximum edges within it are $\frac{1}{2}n_i(n_i - 1)$ (see also \cite{Harary}).\par
In 1997, Albertson in~\cite{ALBERTSON} mention to the imbalance of an edge $uv$ by $imb(uv)$ where $imb(uv)=\lvert d_u-d_v\rvert$.We considered a graph $G$ is regular if all of its vertices have the exact same degree, then the irregularity measure defined in~\cite{ALBERTSON, GUTMAN2, Brandt} as: 
\[
\operatorname{irr}(G)=\sum_{uv\in E(G)}\lvert d_u(G)-d_v(G) \rvert.
\]
 Dorjsembe et al. In~\cite{Dorjsembe} given the relationship between irregularity of Albertson index and minimum, maximum degrees $\delta,\Delta$ of graph $G$ (respectfully), where contribute vital roles in determining connection, shading, component incorporation, and realisability. Among this degrees, which reflect the maximum and minimum number of edges incident to every vertex in the graph, are important elements of a graph's sophistication and performance. The relationship given by: $\operatorname{irr} (G)> \frac{\delta(\Delta-\delta)^2.|V|}{\Delta+1}$, the degree sum formula (or handshaking lemma) is know well as doubles of edges as $2|E|$. In~\cite{Hosam,Brandt} define total irregularity as $\operatorname{irr}_T(G)=\sum_{\{u,v\} \subseteq V(G)}\lvert d_u(G)-d_v(G) \rvert$ and shown upper bounds on the irregularity of graphs. When $G$ is a tree, Ghalavand, A. et al. In~\cite{Ghalavand} proved  $\operatorname{irr}_T(G) \leq \frac{n^2}{4} \operatorname{irr}(G)$ when the bound is sharp for infinitely many graphs.  The first and the second Zagreb index, $M_1(G)$ and $M_2(G)$ are defined in~\cite{TRINAJSTIC, WILCOX,ALBERTSON} as: 
\[
M_1(G)=\sum_{i=1}^{n}d_i^2, \quad \text{and} \quad M_2(G)=\sum_{uv\in E(G)} d_u(G)d_v(G).
\]
Actually, an alternative expressions introduced by M. Matej\'i et al. in~\cite{ALBERTSON, Nikoli} for the first Zagreb index  as $M_1(G)=\sum_{u\sim v}(d_u+d_v)$.  The recently introduced $\sigma(G)$ irregularity index is a simple diversification of the previously established Albertson irregularity index, in~\cite{GutmanTogan,DimitrovAbdo} defined as: 
\[
\sigma(G)=\sum_{uv\in E(G)}\left( d_u(G)-d_v(G) \right)^2.
\]
This paper is organized as follows. In Section~\ref{sec1}, we observe the important concepts to our work including literature view of most related papers, in Section~\ref{sec2} we have provided a preface through some of the important theories and properties we have utilised in understanding the work, in Section~\ref{sec3} we introduced the main result of our paper. The goal of this paper is provide Albertson index, Sigma index among a degree sequence $\mathscr{D}=(d_1,\dots,d_n)$  where it is non-increasing and non-decreasing of tree $T$.
\section{Preliminaries}\label{sec2}
In this section, we study trees with $\mathscr{D}=(d_1,\dots,d_n)$ a degree sequence  where $d_n \geqslant \dots \geqslant d_1$. Lemma~\ref{le2.1} establishes that a graph is connected if it attains the maximal Sigma Index. Let $\Delta(G) = \max\{\deg_G(v) : v \in V(G)\}$ denote the maximum degree of a vertex in $G$. A graph $G$ satisfies $\Delta(G) = n-1$ if at least one vertex has degree $n-1$. Through this paper, we point out that all the trees included in the study are caterpillar trees, we consider caterpillar trees denoted by $\mathscr{C}(n, m)$, where $n$ is the number of backbone (or path) vertices and $m$ is the number of pendant vertices attached to each.

\begin{proposition}\cite{jasem}\label{three.c}
Let $T$ be a tree, let  $\mathscr{D}=(d_1,d_2,d_3)$ be a degree sequence of positive integers where $d_1\geqslant d_2 \geqslant d_3$. The Albertson index among $\mathscr{D}$ given by:
 	\[
 	\irr(T)=\left\lbrace
 	\begin{aligned}
 		& \irr_{max}=(d_1-1)^2+(d_2-1)^2 +(d_3-1)(d_3-2)(d_1-d_3)(d_2-d_3) \\
 		& \irr_{min}=(d_1-1)^2+(d_3-1)^2+(d_2-1)(d_2-2)+(d_1-d_3).
 	\end{aligned} 
 	\right. 
 	\]
 \end{proposition}
 \begin{lemma}\cite{ZhouLin, Clark22}
 Let $p$ be a prime number where $p \geq 1$, let  $\mathscr{D}=(d_1, d_2, \ldots, d_n)$ be a degree sequence where $0\leq d_i\leq n-1$ when $i\in \mathbb{N}$, then the inequality satisfying:
\[
\left(\sum_{i=1}^n d_i^p\right)^{\frac{1}{p}} \leq(n-1)^{1-\frac{1}{p}} \sum_{i=1}^n d_i^{\frac{1}{p}},
\]
if and only if $\left (\sum_{i=1}^{n}\sqrt[p]{d_i}\right)^p \geqslant c\sum_{i=1}^{n}d_i^p$.
 \end{lemma}
 \begin{proposition}~\cite{Gallaipp,Brandt}~\label{proedo}
 Let $\mathscr{D}=(d_1,\dots,d_n)$ be a degree sequence  where $d_1\geqslant d_2\geqslant \dots \geqslant d_n$, then the inequality satisfying: 
 \[
 \sum_{i=1}^{n}d_i\leq n(n-1)+\sum_{i=n+1}^{k} \min(n,d_i).
 \]
 \end{proposition}
 In fact, we employ Proposition~\ref{proedo} for establishes Lemma~\ref{four.c} where we consider $\mathscr{D}=(d_1,d_2,d_3,d_4)$ a degree sequence  systematized this approach, by considering $\irr_T(G)\leq (n-2)\irr(G)$. In Theorem~\ref{hy.1} we refer to caterpillar tree had defined by El-Basil in~\cite{El-Basil} and Oscar Rojo in~\cite{Oscar-Rojo}.
\begin{theorem}{\label{hy.1}}
Let $\mathscr{D}=(d_1,\dots,d_n)$ be a degree sequence with $n\geq 3$ and  let be order as: $d_n > d_1>\dots > d_2 > d_{n-1}$, the caterpillar tree with such order has the  maximum value  of $\irr$ among all caterpillar trees with such degrees sequence of path vertices.
 \end{theorem}
\begin{proposition}
Sigma index for caterpillar tree of order $(n,m)$ given by: 
\[
\sigma(C(n,m))=\begin{cases}
    2m^3, & \quad \text{if } n=2 \\
    2m^3+m-2,  & \quad \text{if } n\geqslant 2. \\ 
\end{cases}
\]
Where $(n,m)$ is mean $n$ vertices, $m$ pendent vertices.
\end{proposition}

\begin{proposition}{\label{caterpillar}}
 Let $T$ be a tree, let $C(n,m)$  be a caterpillar tree with path vertices, let $\mathscr{D}=(d_1,\dots,d_n)$ be a degree sequence, then Albertson index among caterpillar tree given by:
 \[\irr(T)=\left( {{d_n} - 1} \right)^2 + \left( {d_1 - 1} \right)^2 + \sum\limits_{i = 2}^{n - 1} {\left( {{d_i} - 1} \right)\left( {{d_i} - 2} \right)} +\sum_{i=1}^{n-1}|d_i-d_{i+1}|.\]
 \end{proposition}
 We denote by $\delta(G), \Delta(G)$ the minimum and maximum degrees in a graph $G$, denote by $d(G)$ their average degree. 
\begin{definition}
Let $\mathcal{S}$ be a class of graphs, then we have $\operatorname{irr}_{\max},\operatorname{irr}_{\min}$ Albertson index of a graph $G$, where: \begin{gather*} \operatorname{irr}_{\max}(\mathcal{S}) =\max \{\operatorname{irr}(G)\mid G\in \mathcal{S}\}, \\
\operatorname{irr}_{\min}(\mathcal{S}) =\min \{\operatorname{irr}(G)\mid G\in \mathcal{S}\}.
\end{gather*}
 \end{definition}

\begin{lemma}[Yang J, Deng H., M.,~\cite{Yang}]\label{le2.1}
	Let $G$ has the maximal Sigma index among all connected graphs with $n$ vertices and $p$ pendant vertices, where $n, p$ are positive integers such that $1\leq p \leq n-3$. Then: $\Delta(G)=n-1$. 
\end{lemma}
Oboudi, M., R. in~\cite{Oboudi} shows if $G$ is a connected graph then let $\mathscr{D}=(d_1,\dots,d_n)$ be a degree sequence  where $d_1 > d_2> \dots > d_n$ that has $n$ degree graph and $n\geq 3$. The first and second Zagreb indices in~\cite{GutmanTogan} as two additional Zagreb-type indices established the forgotten index as:
\[
F(G)=\sum_{w\in E(G)}\left(d_w\right)^3=\sum_{uv \in E(G)}\left[\left(d_u\right)^2+\left(d_v\right)^2\right] .
\]
\begin{theorem}\cite{Sekar}
Let $ K_{m, n}$ be a complete bipartite graph, then sigma index of  $ K_{m, n}$ given by $\sigma(K_{m, n})=(m n)(n-m)^{2}$.
\end{theorem}
\begin{theorem}[Sigma index ~\label{sigmathm}]\cite{GutmanTogan}
For every simple graph $G$, then Sigma index $\sigma(G)$ is an even integer, so that we have:
\[
\sigma(G)=F(G)-2 M_2(G)=\sum_{uv\in E(G)}\left[\left(d_u\right)^2+\left(d_v\right)^2\right]-2 \sum_{uv \in \in(G)} d_u d_v .
\]
where $M_2(G)=\sum_{uv\in E(G)} d_ud_v$.
\end{theorem}
\begin{theorem}\cite{GutmanTogan}
Let $S_{r, k}$ be a double star graph , the degrees of two adjacent central vertices $u,v$ defined as   $d_{u}=k \geq 3, d_{v}=r \geq 1$. Then the sigma index of $S_{r, k}$ is given by
\[
\sigma\left(S_{r, k}\right)=(k-1)^{3}+k^{2}+(r-1)^{3}+r^{2}-2 k r .
\]
\end{theorem}
\begin{proposition}~\label{proe1}
 Let $\mathscr{G}(n)$ be is the  set of  all $n$ vertex on graph $G$,  $\lambda(G)$ the index  of the largest eigenvalue in graph, and let  $\bar{d}(G)$ be the mean of the vertex degrees, then
 we have: 
\[
\max \{\operatorname{irr}(G): G \in \mathscr{G}(n)\}=\left\{\begin{array}{ll}\frac{1}{4} n-\frac{1}{2} & (n \text { even }) \\ \frac{1}{4} n-\frac{1}{2}+\frac{1}{4 n} & (n \text { odd }).\end{array}\right. 
\].
\end{proposition}
This result implies in~\cite{Bell,Collatz,ALBERTSON, GUTMAN2}. This relationship can be demonstrated for the indices in both Lemma~\ref{lem1}, \ref{lem2}, which depict the star graph.
\begin{lemma}\cite{Nasiri}~\label{lem1}
	The star $S_n$  is the only tree of order n that has a great deal of irregularity, satisfying:
	\[\operatorname{irr}\left( {{S_n}} \right) = \left( {n - 2} \right)\left( {n - 1} \right).\]
\end{lemma} 
\begin{lemma}~\label{lem2}
Let $\mathcal{T}$ be a classes of trees, then $\sigma_{max}, \sigma_{min}$ in trees with $n$ vertices by:
 \[\sigma(\mathcal{T})  = \left\{ \begin{array}{l}
 	{\sigma _{\max }(\mathcal{T})} = \left( {n - 1} \right)\left( {n - 2} \right){\rm{        }};n \ge 3\\
 	{\sigma _{\min }(\mathcal{T})} = 0{\rm{                             ; n = 2}}.
 \end{array} \right.\]
 \end{lemma}
 For any tree $T$ with $n$ vertices has $n-1$ edges, then $\sigma_{\max}(T)=n-2$. Thus, the modified total Sigma in trees as $\sigma_t(T)=\frac{1}{2}\sum_{(u,v)\subseteq V(T)}\left(d_u-d_v\right)^2$.

\section{Main Result}\label{sec3}
In this section, we provide the main result of this paper, we show that in Theorem~\ref{mainthm} by study properties in Lemma~\ref{le.alb4},\ref{le.alb6}. We have $X=(x_1,x_2,\dots,x_n)$ where $n=|X|$, and refer to $m$ a set of leaves adjacent to main vertices. Zhang, P., Wang, X in\cite{j1} show ``caterpillar trees'' are employed in chemical graph theory for describing molecular structures, and several topological indices, which include the Albertson index, have been used to study the characteristics of molecules.

\subsection{Albertson Index With Non- Increasing Degree Sequence}
Let $\mathscr{D}=(d_1,d_2,\dots,d_n)$ be a degree sequence  where $d_1\leqslant d_2 \leqslant \dots d_n$, we show in Hypothesize~\ref{hy.2} the  minimum  value  of Albertson index among all caterpillar trees. In Lemma~\ref{four.c} we discuss a degree sequence of tree according to the first Zagreb index and in Lemma~\ref{le.alb4} we have term $d_4\geqslant d_3 \geqslant d_2 \geqslant d_1$. Also in both of Lemma~\ref{five.c},\ref{le.alb5},\ref{le.alb6} then that hold the main Theorem~\ref{mainthm}.
\begin{hypothesize}{\label{hy.2}}
Let $\mathscr{D}=(d_1,\dots,d_n)$ be a degree sequence with $n\geq 3$ and  let be order as: ${d_n} > d_{n-1}> \dots> d_2>  d_1$, the caterpillar tree with such order has the  minimum  value  of $\irr$ among all caterpillar trees with such degrees sequence of path vertices.
\end{hypothesize}
\begin{proof}
Let be assume the sequence is $d=(d_1,d_2,\dots,d_n)$ where ${d_n} > d_{n-1}> \dots> d_2>  d_1$, then we have: 
\begin{align*}
    \irr(G)&=\sum_{i=1}^{n-1}|d_i-d_{i+1}|+\sum_{i=2}^{n-2}(d_i-1)(d_i-2)+(d_n-1)^2+(d_{n-1}-1)^2\\
    &=\sum_{i=1}^{n-1}|d_i-d_{i+1}|+\sum_{i=2}^{n-2} d_i^2 - 3\sum_{i=2}^{n-2} d_i + 2(n-3) + (d_n^2 - 2d_n + 1) + (d_{n-1}^2 - 2d_{n-1} + 1)\\
    &=\sum_{i=1}^{n-1}|d_i-d_{i+1}|+\sum_{i=2}^{n-2} d_i^2 + d_n^2 + d_{n-1}^2 - 3\sum_{i=2}^{n-2} d_i - 2d_n - 2d_{n-1} + 2(n-3) + 2\\
    &=\sum_{i=1}^{n-1}|d_i-d_{i+1}|+\sum_{i=2}^{n-2} d_i^2 + d_n^2 + d_{n-1}^2 - 3\sum_{i=2}^{n-2} d_i - 2d_n - 2d_{n-1} + 2n - 4\\
    &=\sum_{i=1}^{n-1}|d_i-d_{i+1}|+\sum_{i=2}^{n-2} d_i^2- 3\sum_{i=2}^{n-2} d_i+d_n^2 + d_{n-1}^2- 2(d_n +d_{n-1}) + 2n - 4.
\end{align*}
where we have $d_n+d_{n-1}$ is the minimum when ${d_n} > d_{n-1}> \dots> d_2>  d_1$, then we have the cases are: 
\case{1} we know for $\sum_{i=2}^{n-2} d_i^2,  3\sum_{i=2}^{n-2} d_i$, if $d_i$ are all 1, then $\sum d_i^2 = n-3$ and $3\sum d_i = 3(n-3)$, so in this case, $(d_i)^2$ is the minimum. Also, we have from this inequity $ d_n^2 + d_{n-1}^2$ is still minimum when ${d_n} \geqslant d_{n-1}\geqslant \dots \geqslant d_2 \geqslant  d_1$.
\case{2} For $(d_i)^2, |d_i-d_{i+1}|$, we know if $d_i>d_{i+1}$, then $d_i^2 \geq (d_i - d_{i+1})$, then we have: 
\begin{align*}
    \irr(G)&=(d_1 - d_n) + \sum_{i=2}^{n-2} d_i^2 - 3\sum_{i=2}^{n-2} d_i + (d_n - 1)^2 + (d_{n-1} - 1)^2 + 2n - 6\\
    &=d_1 + \sum_{i=2}^{n-2} d_i^2 - 3\sum_{i=2}^{n-2} d_i + d_{n-1}^2 - 2d_{n-1} + 1 + d_n^2 - 2d_n + 1 + 2n - 6\\
    &= d_1 + \sum_{i=2}^{n-2} d_i^2 - 3\sum_{i=2}^{n-2} d_i + d_n^2 + d_{n-1}^2 - 2(d_n + d_{n-1}) + 2n - 4.
\end{align*}
Let $\mathscr{D}=(d_1,\dots,d_n)$ be a degree sequence  where $d_n>d_1>d_{n-2}>d_3\dots>d_2>d_{n-1}$ or 
If $d_i\leq d_{i+1}$, then we have $d_{i+1}$ as $d_i + k$ where $k > 0$. where $\min \{d_n,d_1\}=d_n, (d_n-1)^2=d_n^2 - 2d_n + 1$, thus $(d_i)^2>|d_i-d_{i+1}|$. Therefore we have: 
\begin{align*}
    \irr(G)&=\sum_{i=2}^{n-2} d_i^2 - 3\sum_{i=2}^{n-2} d_i + (d_n - 1)^2 + (d_{n-1} - 1)^2+d_n - d_1 + 2n - 6\\
    &=\sum_{i=2}^{n-2} d_i^2 - 3\sum_{i=2}^{n-2} d_i + d_n^2 + d_{n-1}^2 - 2(d_n + d_{n-1})-d_1 + 2n - 4.
\end{align*}
Thus, from both the previously discussed cases 1 and 2, we observed that the relation is true for $n$ and therefore the relation has the minimum value of the Albertson index when ${d_n} \geqslant d_{n-1}\geqslant \dots \geqslant d_2 \geqslant  d_1$ and now we need to prove it for $n+1$ where we consider both $\sum_{i=2}^{n-2}(d_i-1)(d_i-2)$ as constants and $d_n+d_{n-1}$ as the minimum value and to prove that the relation $\sum_{i=2}^{n}|d_i-d_{i+1}|$ is as minimum as possible for the Albertson index according to the following:
\begin{align*}
    \irr(G)&=\sum_{i=1}^{n-1}|d_i-d_{i+1}|+\sum_{i=2}^{n-2}(d_i-1)(d_i-2)+(d_n-1)^2+(d_{n-1}-1)^2\\
    &=\sum_{i=1}^{n-1} |d_i - d_{i+1}| + \sum_{i=2}^{n-2} d_i^2 - 3\sum_{i=2}^{n-2} d_i + d_n^2 + d_{n-1}^2 - 2(d_n + d_{n-1}) + 2n - 4.
\end{align*}
Clarity, we notice $\sum_{i=2}^{n-2} d_i^2 - 3\sum_{i=2}^{n-2} d_i + d_n^2 + d_{n-1}^2 - 2(d_n + d_{n-1}) + 2n - 4$ is the minimum when ${d_n} \geqslant d_{n-1}\geqslant \dots \geqslant d_2 \geqslant  d_1$, then we have: 
\begin{itemize}
    \item If $d_i\geq d_{i+1}$, then we have $\irr(G)=d_1 + \sum_{i=2}^{n-2} d_i^2 - 3\sum_{i=2}^{n-2} d_i + d_n^2 - 3d_n + d_{n-1}^2 - 3d_{n-1} + 2n - 5$.
    \item If $d_i\leq d_{i+1}$, then we have: $\irr(G)=-d_1 + \sum_{i=2}^{n-2} d_i^2 - 3\sum_{i=2}^{n-2} d_i + d_n^2 - d_n + d_{n-1}^2 - d_{n-1} + 2n - 5$.
\end{itemize}
From this cases notice $\sum_{i=2}^{n-2} d_i^2 - 3\sum_{i=2}^{n-2} d_i+d_{n-1}^2+d_{n}^2+2n - 5$ is minimum according to the constant term,  then $\sum_{i=1}^{n-1}|d_i-d_{i+1}|+\sum_{i=2}^{n-2}(d_i-1)(d_i-2)+(d_n-1)^2+(d_{n-1}-1)^2$ is minimum  when ${d_n} \geqslant d_{n-1}\geqslant \dots \geqslant d_2 \geqslant  d_1$.
As desire.
\end{proof}
Actually, we provide in Proposition~\ref{tes.1} the inequality of maximum and minimum degree in tree (see Lemma~\ref{lem1}) according to the first Zagreb index as we show that in Lemma~\ref{four.c}.
\begin{proposition}~\label{tes.1}
Let $T$ be a tree of order $n=4$, let $\mathscr{D}=(d_1,\dots,d_4)$ be a degree sequence  where $d_1 \geq d_2 \geq d_3 \geq d_4$, then we have: 
\[
\irr(T)< 2 \sqrt{\Delta M_1(T)}+\delta.
\]
\end{proposition}
\begin{lemma}\label{four.c}
Let $T$ be a tree, let $\mathscr{D}=(d_1,\dots,d_4)$ be a degree sequence  where $d_1 \geq d_2 \geq d_3 \geq d_4$, then we have:
\[
\irr(T)=M_1(T)^2-2\sqrt{M_1(T)}+\sum_{i=1}^4\left|x_i-x_{i+1}\right|-(d_2+d_3)-1.
\]
\end{lemma}

\begin{proof}
By using Albertson index we define the sequence $d=(d_1,d_2,d_3,d_4)$ where $d_1 \geq d_2 \geq d_3 \geq d_4$, according to Proposition~\ref{tes.1} we define for numbers $a, b, c, d$, then by considering $M_1(G)=\sum_{i=1}^{n}d_i^2$ we discus according to Proposition~\ref{proedo} the following cases.
\case{1} In this case, we have the largest value $a,d$, the smallest value $b,c$, then we have: 
\begin{align*}
    \irr(T_1)&= (a-b)+|b-c|+(d-c)-(b+c)\\
	& =(a+d)-(b+c)+|b-c|-(b+c) \\
	& =(a+d)-2(b+c)+|b-c|\\
        &=|a-b|+|b-c|+|c-d|-(b+c)\quad \text{where } |b-c|=-2 d_3-2 d_4\\
        &=\left(d_1+d_2\right)-2\left(d_3+d_4\right) \\
        & =d_1+d_2-2 d_3-2 d_4.
\end{align*}
\case{2} In this case, we have the largest value $b,c$, the smallest value $a,d$, then we have: 
\begin{align*}
    \irr(T_2)&= b-a+|b-c|+c-d-(b+c)\\
	& =B-L-B+|b-c|=|b-c|-L \quad \text{where } B=\sum (a,c), L=\sum (a,d), |b-c|=d_1-d_2 \\
	& =d_1-d_2-d_3-d_4.
\end{align*}
We noticed that from both cases $\irr(T_1)+\irr(T_2)<\irr(T_1)-\irr(T_2)$, then we have $\max(\irr(T_1),\irr(T_2))=2d_2-d_3-d_4$ and $\min(\irr(T_1),\irr(T_2))=2d_1-3d_3-3d_4$.
\case{3} In this case, the largest value is $a,b$, the smallest value is $d,c$, then we have: 
\begin{align*}
\irr(T_3)&=|a-b|+b-c+|c-d|-b-c\\
&=d_1-d_2+d_3-d_4-2 C\quad \text{where } C \in\left\{d_3, d_4\right\}\\
&=d_1-d_2-d_3-d_4 \quad \text { if } c=d_3\\
&=d_1-d_2+d_3-3 d_4 \quad \text { if } c=d_4.
\end{align*}

\case{4} In this case, the largest value is $a,c$, the smallest value is $d,b$, then we have: 
\begin{align*}
\irr(T_4)&=a-b+c-b+c-d-b-c\\
&=a-3 b+c-d\\
&=d_1-d_2-3 d_3-d_4 \quad \text { if } b=d_3\\
&= d_1+d_2-d_3-3 d_4 \quad \text { if } b=d_4.
\end{align*}
Therefore, from cases $\irr(T_1),\irr(T_2), \irr(T_3), \irr(T_4)$ we have after comparing 
\[
\irr(T)=\begin{cases}
    d_1+d_2-d_3-3 d_4: \quad \text { the biggest } \\
	d_1-d_2-d_3-d_4: \quad \text { the smallest }.
\end{cases}
\]
We should be know: $a \in\left\{d_1, d_2\right\}, b \in\left\{d_3, d_4\right\}, c \in\left\{d_3, d_4\right\}, d \in\left\{d_1, d_2\right\}$ \\
So that we have:
$2 d_2-d_3 \geq 0$; so, we obtain on:
If $a, d \in\left\{d_1, d_2\right\}$ and $b, c \in\left\{d_3, d_4\right\}$ or $a, c \in\left\{d_1, d_2\right\}$ and $b=d_4, d=d_3$, we
have: 
\[
\irr_{\max}(T)=\sum_{i=1}^4\left(d_i-1\right)^2+d_1+d_2-d_3-3 d_4+2.
\]
And we have also:
If $a, d \in\left\{d_3, d_4\right\}$ and $b, c \in\left\{d_1, d_2\right\}$ or $a, \mathrm{~b} \in\left\{d_1, d_2\right\}$ and $\mathrm{d}=d_4, \mathrm{c}=d_3$, we have \\
\[
\irr_{\min}(T)=\sum_{i=1}^4\left(d_i-1\right)^2+d_1-d_2-d_3-d_4+2.
\]
Therefore, we have
\begin{align*}
    \irr(G)&=(a-1)^2+(b-2)(b-1)+(c-2)(c-1)+(d-1)^2+|a-b|+|b-c|+|c-d| \\
		& =(a-1)^2+(b-1)^2-(b-1)+(c-1)^2-(c-1)+(d-1)^2+(a-b)+(b-c)+|c-d|\\
        &=\sum_{x \in\{a, b, d\}}(x-1)^2+|a-b|+|b-c|+|c-d|-(b+c)+2
\end{align*}
Finally, according to the largest value and smallest value, we organized that as the maximum value and minimum value, thus we have: 
\[
\irr(T)=\begin{cases}
    & \irr_{\max}(T)=\sum_{i=1}^4\left(d_i-1\right)^2+d_1+d_2-d_3-3 d_4+2 \\
	& \irr_{\min}(T)=\sum_{i=1}^4\left(d_i-1\right)^2+d_1-d_2-d_3-d_4+2.
\end{cases}
\]
As desired.
\end{proof}
Actually, these assertions quantify how the arrangement of degree sequence $\mathscr{D}$ along the path in a caterpillar  tree $C(n,m)$ where $n$ vertices and $m$ pendent vertices  effects structural indices, lending credence to the notion that a given ordering produces extremal (in this case, both of Lemma~\ref{le.alb4},\ref{le.alb5}) values for non-decreasing degree sequence.
\begin{lemma}~\label{le.alb4}
Let $T$ be a tree of order $n\geqslant4$, let $\mathscr{D}=(d_1,\dots,d_4)$ be a non-decreasing degree sequence where $d_4\geqslant d_3 \geqslant d_2 \geqslant d_1$, then Albertson index among tree $T$ is: 
\[
\irr(T)=\begin{cases}
    \irr_{\max}(T)=d_1^2+d_2^2+\sum_{i=1}^{3}\lvert d_i-d_{i+2}\rvert+d_3^2+d_4^2+d_3+d_4-6 \\
    \irr_{\min}(T)=d_3^2+d_4^2+\sum_{i=1}^{2}\lvert d_i-d_{i+2}\rvert+\lvert d_1-d_2\rvert+d_1^2+d_2^2+d_1+d_2-6.
\end{cases}
\]
\end{lemma}
\begin{lemma}\label{five.c}
Let $T$ be a tree,  let $\mathscr{D}_i=(d_1,\dots,d_5)$ be a non-increasing degree sequence  where $d_1\geqslant d_2 \geqslant \dots \geqslant d_5$ and let $\mathscr{D}_j=(d_1,\dots,d_6)$ be a non-increasing degree sequence where $d_1\geqslant \dots \geqslant d_6$ where $i,j>0$, Albertson index among tree $T$ is: 
\[
\irr(T_{\mathscr{D}_i})=\sum_{i=1}^{5} (d_i-1)(d_i-2)+\sum_{i=1}^{5} |d_i-d_{i+1}|+(d_1-1)^2 +(d_5-1)^2.
\]
and 
\[
\irr(T_{\mathscr{D}_j})=\sum_{i=1}^{6} (d_i-1)(d_i-2)+\sum_{i=1}^{6} |d_i-d_{i+1}|+(d_1-1)^2 +(d_6-1)^2.
\]
\end{lemma}
\begin{lemma}~\label{le.alb5}
Let $T$ be tree of order $n\geqslant5$, let $\mathscr{D}=(d_1,\dots,d_5)$ be a non-decreasing degree sequence where $d_5\geqslant d_4\geqslant d_3 \geqslant d_2 \geqslant d_1$, then Albertson index among tree $T$ is: 
\[
\irr(T)=d_1^2+d_n^2+\sum_{i=2}^{n-1}\lvert d_i-d_{i+1}\rvert+\sum_{i=2}^{4}(d_i+2)(d_i-1)-2.
\]
\end{lemma}
\begin{proof}
According to Lemma~\ref{le.alb4}, we will employed that for a degree sequence $\mathscr{D}=(d_1,d_2,d_3,d_4,d_5)$ where $d_5\geqslant d_4\geqslant d_3 \geqslant d_2 \geqslant d_1$, therefore we will discus following cases by considering the constant term is $\sum_{i=2}^{4}(d_i+2)(d_i-1)$.
\case{1} In this case, the largest value is $d_1,d_5$, the smallest value is $d_2,d_3,d_4$, then we have following cases.
\begin{itemize}
    \item In this case, the sequence is $\mathscr{D}=(d_1,d_2,d_3,d_4,d_5)$, then we have 
\begin{align*}
\irr(T_1)&=(d_1+1)(d_1-1)+(d_5+1)(d_5-1)+\lvert d_1-d_2\rvert+\lvert d_2-d_3\rvert+\lvert d_3-d_4\rvert+\lvert d_4-d_5\rvert+\\
&+(d_2+2)(d_2-1)+(d_3+2)(d_3-1)+(d_4+2)(d_4-1)\\
&=d_1^2+d_5^2+\sum_{i=1}^{4}\lvert d_i-d_{i+1}\rvert +\sum_{i=2}^{4}(d_i+2)(d_i-1)-2.
\end{align*}

 \item In this case, the sequence is $\mathscr{D}=(d_1,d_3,d_2,d_4,d_5)$, the constant term is $\lvert d_3-d_2\rvert+\lvert d_4-d_5\rvert-2$, then we have, 
\begin{align*}
\irr(T_2)&=(d_1+1)(d_1-1)+(d_5+1)(d_5-1)+\lvert d_1-d_3\rvert+\lvert d_3-d_2\rvert+\lvert d_2-d_4\rvert+\lvert d_4-d_5\rvert+\\
&+\sum_{i=2}^{4}(d_i+2)(d_i-1)\\
&=d_1^2+d_5^2+\sum_{i\in\{1,2\}} \lvert d_i-d_{i+2}\rvert+\lvert d_3-d_2\rvert+\lvert d_4-d_5\rvert+\sum_{i=2}^{4}(d_i+2)(d_i-1)-2.
\end{align*}

 \item In this case, the sequence is $\mathscr{D}=(d_1,d_3,d_4,d_2,d_5)$, the constant term is $\lvert d_3-d_4\rvert+\lvert d_2-d_5\rvert-2$, then we have, 
\begin{align*}
\irr(T_3)&=(d_1+1)(d_1-1)+(d_5+1)(d_5-1)+\lvert d_1-d_3\rvert+\lvert d_3-d_4\rvert+\lvert d_4-d_2\rvert+\lvert d_2-d_5\rvert+\\
&+\sum_{i=2}^{4}(d_i+2)(d_i-1)\\
&=d_1^2+d_5^2+\sum_{i\in\{1,2\}}\lvert d_i-d_{i+2}\rvert+\lvert d_3-d_4\rvert+\lvert d_2-d_5\rvert+\sum_{i=2}^{4}(d_i+2)(d_i-1)-2\\
&=d_1^2+d_5^2+\sum_{i\in\{1,2\}}\lvert d_i-d_{i+2}\rvert-d_1 + 2d_4 - 2d_2 + d_5+\sum_{i=2}^{4}(d_i+2)(d_i-1)-2.
\end{align*}
\end{itemize}
\case{2} In this case, the largest value is $d_2$, the smallest value is $d_1,d_3,d_4,d_5$, then we have following cases.
\begin{itemize}
    \item In this case, the sequence is $\mathscr{D}=(d_2,d_3,d_1,d_5,d_4)$, then we have
    \begin{align*}
    \irr(T_4)&=d_2^2+d_4^2+|d_2-d_3|+|d_3-d_1|+|d_1-d_5|+|d_5-d_4|+\sum_{i\in\{1,3,5\}}(d_i+2)(d_i-1)-2\\
    &=d_2^2+d_4^2+2d_5 + 2d_3 - 2d_1 - d_2 - d_4+\sum_{i\in\{1,3,5\}}(d_i+2)(d_i-1)-2.
    \end{align*}
\item In this case, the sequence is $d=(d_2,d_5,d_3,d_1,d_4)$, then we have
\begin{align*}
\irr(T_5)&=d_2^2+d_4^2+|d_2-d_5|+|d_5-d_3|+|d_3-d_1|+|d_1-d_4|+\sum_{i\in\{1,3,5\}}(d_i+2)(d_i-1)-2\\
    &=d_2^2+d_4^2+2d_5 - 2d_1 - d_2 + d_4+\sum_{i\in\{1,3,5\}}(d_i+2)(d_i-1) - 2.
\end{align*}
\item In this case, the sequence is $\mathscr{D}=(d_2,d_1,d_5,d_3,d_4)$, then we have
\begin{align*}
\irr(T_6)&=d_2^2+d_4^2+|d_2-d_1|+|d_1-d_5|+|d_5-d_3|+|d_3-d_4|+\sum_{i\in\{1,3,5\}}(d_i+2)(d_i-1)-2\\
    &=d_2^2+d_4^2+2d_5 - 2d_1 + d_2 - 2d_3 + d_4+\sum_{i\in\{1,3,5\}}(d_i+2)(d_i-1) - 2.
\end{align*}
\end{itemize}
\case{3} In this case, the largest value is $d_3,d_5$, the smallest value is $d_1,d_2,d_4$, then we have following cases.
\begin{itemize}
    \item In this case, given the sequence as $\mathscr{D}=(d_3,d_1,d_2,d_4,d_5)$, then we have
    \begin{align*}
        \irr(T_7)&=d_3^2 + d_5^2 - 2d_1 + d_3 + d_5+\sum_{i\in\{1,2,4\}}(d_i+2)(d_i-1) - 2\\
        &=d_3^2+d_5^2+\sum_{i\in \{1,2\}}\lvert d_i-d_{i+2}\rvert+\sum_{i\in \{1,4\}}\lvert d_i-d_{i+1}\rvert+\sum_{i\in\{1,2,4\}}(d_i+2)(d_i-1)-2.
    \end{align*}
   \item In this case, given the sequence as $\mathscr{D}=(d_3,d_2,d_4,d_1,d_5)$, then we have
    \begin{align*}
        \irr(T_8)&=d_3^2 + d_5^2 - 2d_1 - 2d_2 + d_3 + 2d_4 + d_5+\sum_{i\in\{1,2,4\}}(d_i+2)(d_i-1) - 2\\
        &=d_3^2+d_5^2+|d_3-d_2|+|d_2-d_4|+|d_4-d_1|+|d_1-d_5|+\sum_{i\in\{1,2,4\}}(d_i+2)(d_i-1)-2.
    \end{align*}
\item In this case, the sequence is $\mathscr{D}=(d_3,d_4,d_1,d_2,d_5)$, we notice that in this case the constant term is $|d_4-d_1|+|d_2-d_5|-2$, then we have
\begin{align*}
    \irr(T_9)&=d_3^2 + d_5^2+|d_3-d_4|+|d_4-d_1|+|d_1-d_2|+|d_2-d_5|+\sum_{i\in\{1,2,4\}}(d_i+2)(d_i-1)-2\\
    &=d_3^2 + d_5^2+\sum_{i\in\{1,3\}} \lvert d_i-d_{i+1}\rvert+|d_4-d_1|+|d_2-d_5|+\sum_{i\in\{1,2,4\}}(d_i+2)(d_i-1)-2.
\end{align*}
\end{itemize}
Actually, respecting the constant terms in each case that shift with the shift of vertices, we notice in \textbf{Case 1} for $\irr(T_2), \irr(T_3)$, the maximum value is $\irr(T_3)$ and the minimum value is $\irr(T_2)$, is available for each case, another aspect to note is that there are repeated results by comparing the vertices with each other, so we restrict ourselves to these comparisons, which they have distributed through three main cases that give us the maximum and minimum values in this sequence, considering the constant terms in each of them. Thus, we have: 
\[
\irr(T)=\begin{cases}
    \irr_{\max}(T)= d_3^2 + d_5^2 - 2d_1 - 2d_2 + d_3 + 2d_4 + d_5+\sum_{i\in\{1,2,4\}}(d_i+2)(d_i-1) - 2\\
    \irr_{\min}(T)=d_2^2+d_4^2+2d_5 + 2d_3 - 2d_1 - d_2 - d_4+\sum_{i\in\{1,3,5\}}(d_i+2)(d_i-1)-2.
\end{cases}
\]
As desire.
\end{proof}
\begin{lemma}~\label{le.alb6}
Let $T$ be is a tree of order $n$, let $\mathscr{D}=(d_1,\dots,d_6)$ be a non-decreasing degree sequence where $d_6\geqslant \dots \geqslant d_1$, then Albertson index over tree $T$ is given by: 
\[
\irr(T)=\begin{cases}
    \irr_{\max}(T)= d_1^2+d_6^2+\sum_{i=1}^{6}\lvert d_i-d_{i+1}\rvert+\sum_{i=2}^{5}( d_i+2)( d_i-1)-2  \\
    \irr_{\min}(T)=d_1^2+d_2^2+\sum_{i=1}^{6}\lvert d_i-d_{i+1}\rvert+\sum_{i\in \{3,4,5,6\}}^{ }( d_i+2)( d_i-1)-2.
\end{cases}
\]
\end{lemma}
According to the results we gave in both Lemmas~\ref{le.alb6}, \ref{le.alb5} and \ref{le.alb4}, we employ them in Theorem~\ref{mainthm} for degree sequence is $d=(d_1,\dots,d_n)$ where $d_n\geqslant \dots \geqslant d_1$ as we show that.

\begin{theorem}~\label{mainthm}
    Let $T$ be tree of order $n$, let $\mathscr{D}=(d_1,\dots,d_n)$ be a non-decreasing degree sequence where $d_n\geqslant \dots \geqslant d_1$, then Albertson index of tree $T$ is: 
    \[
    \irr(T)=d_1^2+d_n^2+\sum_{i=2}^{n-1} d_i^2+\sum_{i=2}^{n-1} d_i+d_n - d_1-2n-2.
    \]
\end{theorem}
\begin{proof}
 Let be assume the sequence is $\mathscr{D}=(d_1,\dots,d_n)$ where $d_n\geqslant \dots \geqslant d_1$, then by considering measures the total absolute difference between consecutive terms are $\sum_{i=1}^{n-1} (d_{i+1} - d_i) = (d_2 - d_1) + (d_3 - d_2) + \dots + (d_n - d_{n-1}) = d_n - d_1$. In performing the proof of the relationship, we use mathematical induction, where we find that the relationship is clearly true for certain values, so we will prove the validity of the relationship for $n$ and then for $n+1$.
 \case{1} Assume the relationship is correct for some value and we need to prove that for $n$ where $n\geq 3$, then we have
 \begin{align*}
     \irr(T)&=\sum_{i=1}^{n}\lvert d_i-d_{i+1}\rvert +(d_1^2-1)+(d_n^2-1)+\sum_{i=2}^{n-1}\lvert d_i+2\rvert \lvert d_i-1\rvert\\
     &=(d_n - d_1) + (d_1^2 - 1) + (d_n^2 - 1) + \sum_{i=2}^{n-1} (d_i + 2)(d_i - 1)\\
     &= d_1^2+d_n^2+\sum_{i=2}^{n-1} d_i^2+\sum_{i=2}^{n-1} d_i+d_n - d_1-2(n-2)-2.
 \end{align*}
 It is clear to show $\sum_{i=2}^{n-1} d_i^2+\sum_{i=2}^{n-1} d_i >d_1^2+d_n^2$ by considering the constant term is $d_n - d_1-2n+2$, in this case, there are multiple approaches that tackle irregularity have
certain of those measurements are in concepts, then by considered the quantity $\sum_{i=2}^{n-1} d_i^2+\sum_{i=2}^{n-1} d_i\leqslant (n-1)(n-1)^2+2$, 
 So that the relationship is correct for $n$ when $d_n\geqslant \dots \geqslant d_1$.
 \case{2} Now we need to prove that for $n+1$, by considering $\sum_{i=2}^{n} (d_{i+1} - d_i) = (d_3 - d_2) + (d_4 - d_3) + \dots + (d_{n+1} - d_n) = d_{n+1} - d_2$, The expression holds for $n\geq 3$, then we have 
 \begin{align*}
     \irr(T)&=\sum_{i=2}^{n+1}\lvert d_i-d_{i+1}\rvert +(d_2^2-1)+(d_{n+1}^2-1)+\sum_{i=3}^{n}\lvert d_i+2\rvert \lvert d_i-1\rvert\\
     &=(d_{n+1} - d_2) + (d_2^2 - 1) + (d_{n+1}^2 - 1) + \sum_{i=3}^{n} |d_i + 2| |d_i - 1|\\
     &=d_{n+1}^2+d_2^2+\sum_{i=3}^{n}d_i^2+\sum_{i=3}^{n}d_i+d_{n+1} - d_2-2(n-2).
 \end{align*}
 Therefore the constant term is $(d_{n+1} - d_2) + (d_2^2 - 1) + (d_{n+1}^2 - 1)$ and $(d_n - d_1) + (d_1^2 - 1) + (d_n^2 - 1)$ it is hold to same result in \textbf{Case 1}, thus the relation is correct for $n+1$.
 So from both Cases 1 and 2, we find that the relationship is realised.

 As desire.
\end{proof}
\subsection{Sigma Index With Non-Decreasing Degree Sequence}

\begin{lemma}~\label{le.segma3}
Let $T$ be tree of order $n\leqslant3$, let $\mathscr{D}=(d_1,d_2,d_3)$  be a degree sequence  where $d_3\geqslant d_2 \geqslant d_1$, then Sigma index is: 
\[
\sigma(T)=\begin{cases}
    \sigma_{\max}(T)= \sum_{i\in\{2,3\}}(d_i+1)(d_i-1)^2+(d_3-d_1)^2+(d_1-d_2)^2\\
    \sigma_{\min}(T)=\sum_{i\in\{1,3\}}(d_i+1)(d_i-1)^2+(d_1-d_3)^2+(d_3-d_2)^2.
\end{cases}
\]
\end{lemma}

Furthermore, in Lemma~\ref{le.sigma2} we provide $\mathscr{D}=(d_1,d_2,d_3, d_4)$ a degree sequence where $d_4\geqslant\dots\geqslant d_1$ to define Sigma index among tree with determine sequence as we show that in Lemma~\ref{le.segma3} where the vertices is located as random and by respectively the largest and smallest vertex.
 
 \begin{lemma}~\label{le.sigma2}
Let $T$ be a tree of order $n>0$ and let $\mathscr{D}=(d_1,d_2,d_3,d_4)$  be a degree sequence where $d_4\geqslant d_3\geqslant d_2\geqslant d_1$, then Sigma index is: 
\[
\sigma(T)=\begin{cases}
    \sigma_{\max}(T)= \sum_{i\in \{2,3\}}(d_i+1)(d_i-1)^2+\sum_{i\in \{1,4\}}(d_i+2)(d_i-1)^2+\sum_{i=1}^{4}(d_i-d_{i+2})^2+(d_4-d_1)^2\\
    \sigma_{\min}(T)=\sum_{i\in \{1,4\}}(d_i+1)(d_i-1)^2+\sum_{i=1}^{4}(d_i-d_{i+2})^2+\sum_{i\in \{2,3\}}(d_i+2)(d_i-1)^2.
\end{cases}
\]
\end{lemma}
According to Lemma~\ref{le.sigma2},\ref{le.segma3}, in Hypothesize~\ref{hy.sigma5}, we present a preliminaries for the sequence $d=(d_1,d_2,d_3,d_4,d_5)$ where $d_5\geqslant d_4 \geqslant d_3 \geqslant d_2 \geq d_1$ as we show that.
\begin{hypothesize}~\label{hy.sigma5}
Let $T$ be a tree of order $n>0$, and let $\mathscr{D}=(d_1,\dots,d_5)$  be a degree sequence where $d_5\geqslant d_4 \geqslant d_3 \geqslant d_2 \geq d_1$, then Sigma index of $T$ given by: 
\[
\sigma(T)=\sum_{i=1}^{3}d_i(d_{i+1})^2+(d_1-1)^3+(d_4)^3 +\sum_{i=1}^{4}(d_i-d_{i+1})^2.
\]
where inequality holds if and only if $d_i=d_{i-1}+1$.
\end{hypothesize}
\begin{proof}
According to the sequence $\mathscr{D}=(d_1,\dots,d_5)$ where $d_5\geqslant d_4 \geqslant d_3 \geqslant d_2 \geq d_1$, we have many cases of position the vertices in tree $T$. In fact, the order in which the elements of this sequence can be realised is 120, and in order to choose the maximum and minimum cases, we get 10 cases, which we can demonstrate as follows: 
\case{1} in this case, we have the sequence is $\mathscr{D}=(d_1,d_2,d_3,d_4,d_5)$, then we have: 
\begin{align*}
    \sigma(T)&=(d_1-1)^3+(d_1-1)(d_1)^2+d_1(d_1+1)^2+d_2(d_3)^2+(d_4)^3+4\\
    &=3d_1^3 - 2d_1^2 + 4d_1 + d_2d_3^2 + d_4^3 + 3.
\end{align*}
\case{2} in this case, we have the sequence is $\mathscr{D}=(d_1,d_2,d_3,d_5, d_4)$, then we have: 
\begin{align*}
    \sigma(T)&=(d_1-1)^3+(d_1-1)(d_1)^2+d_1(d_2)^2+d_3(d_4)^2+(d_2)^3+4\\
    &=2d_1^3 - 4d_1^2 + 3d_1 + d_1d_2^2 + d_3d_4^2 + d_2^3 + 3.
\end{align*}
\case{3} in this case, we have the sequence is $\mathscr{D}=(d_1,d_2,d_5,d_3, d_4)$, then we have: 
\begin{align*}
    \sigma(T)&=(d_1-1)^3+(d_1-1)(d_1)^2+d_3(d_4)^2+d_1(d_2)^2+(d_3)^3+15\\
    &=2d_1^3 - 4d_1^2 + 3d_1 + d_1d_2^2 + d_3d_4^2 + d_3^3 + 14.
\end{align*}
\case{4} in this case, we have the sequence is $\mathscr{D}=(d_1,d_5,d_2,d_3, d_4)$, then we have: 
\begin{align*}
    \sigma(T)&=(d_1-1)^3+(d_3)(d_4)^2+(d_1-1)(d_1)^2+d_1(d_2)^2+(d_3)^3+20\\
    &=2d_1^3 - 4d_1^2 + 3d_1 + d_1d_2^2 + d_3d_4^2 + d_3^3 + 19.
\end{align*}
\case{5} in this case, we have the sequence is $\mathscr{D}=(d_1,d_3,d_4,d_5, d_2)$, then we have: 
\begin{align*}
    \sigma(T)&=(d_1-1)^3+d_1(d_1+1)^2+d_2(d_3)^2+d_3(d_4)^2+(d_1)^3+15\\
    &=3d_1^3 - d_1^2 + 4d_1 + d_2d_3^2 + d_3d_4^2 + 14.
\end{align*}
\case{6} in this case, we have the sequence is $\mathscr{D}=(d_5,d_1,d_2,d_3, d_4)$, then we have: 
\begin{align*}
    \sigma(T)&=(d_4)^3+(d_1-2)(d_1-1)^2+(d_1-1)(d_1)^2+d_1(d_2)^2+(d_3)^2+12\\
    &=2d_1^3 - 5d_1^2 + 5d_1 + d_1d_2^2 + d_3^2 + d_4^3 + 10.
\end{align*}
\case{7} in this case, we have the sequence is $\mathscr{D}=(d_5,d_1,d_4,d_2, d_3)$, then we have: 
\begin{align*}
    \sigma(T)&=(d_4)^3+(d_1-2)(d_1-1)^2+(d_2)(d_3)^2+(d_1-1)(d_1)^2+(d_2)^3+23\\
    &=2d_1^3 - 5d_1^2 + 5d_1 + d_2d_3^2 + d_2^3 + d_4^3 + 21.
\end{align*}
\case{8} in this case, we have the sequence is $\mathscr{D}=(d_1,d_5,d_2,d_4,d_3)$, then we have: 
\begin{align*}
    \sigma(T)&=(d_1-1)^3+(d_3)(d_4)^2+(d_1-1)(d_1)^2+(d_2)(d_3)^2+(d_2)^3+23\\
    &=2d_1^3 - 4d_1^2 + 3d_1 + d_2d_3^2 + d_3d_4^2 + d_2^3 + 22.
\end{align*}
\case{9} in this case, we have the sequence is $\mathscr{D}=(d_5,d_4,d_1,d_3,d_2)$, then we have: 
\begin{align*}
    \sigma(T)&=(d_4)^3+(d_2)(d_3)^2+(d_1-2)(d_1-1)^2+(d_1)(d_2)^2+(d_1)^3+10\\
    &=2d_1^3 - 4d_1^2 + 5d_1 + d_1d_2^2 + d_2d_3^2 + d_4^3 + 8.
\end{align*}
\case{10} in this case, we have the sequence is $\mathscr{D}=(d_5,d_2,d_4,d_1,d_3)$, then we have: 
\begin{align*}
    \sigma(T)&=(d_4)^3+(d_1-1)(d_1)^2+(d_2)(d_3)^2+(d_1-2)(d_1-1)^2+(d_2)^3+26\\
    &=2d_1^3 - 5d_1^2 + 5d_1 + d_2d_3^2 + d_2^3 + d_4^3 + 24.
\end{align*}

To analyse these cases, we will debate each case with the others and then extract the maximum case value and demonstrate it through the following discussion as: 
\begin{itemize}
    \item From \textbf{Case 1}, we have $\sigma(T)=d_1^3 + 2d_1^2 + d_1 + d_2d_3^2 - d_1d_2^2 + d_4^3 - d_3d_4^2 - d_2^3$,
    \item From \textbf{Case 2}, we have $\sigma(T)=-d_1^3 - 3d_1^2 - d_1 + d_1d_2^2 - d_2d_3^2 + d_2^3 - 11$ if and only if $d_1\geqslant 3$,
    \item From \textbf{Case 3}, we have $\sigma(T)=d_1^2 - 2d_1 + d_3d_4^2 - d_3^2 + d_3^3 - d_4^3 + 4$,
    \item From \textbf{Case 4}, we have $\sigma(T)= d_1^2 - 2d_1 + d_3d_4^2 - d_3^2 + d_3^3 - d_4^3 + 9$,
    \item  From \textbf{Case 5}, we have $\sigma(T)= d_1^3 + 4d_1^2 - d_1 + d_2d_3^2 - d_1d_2^2 + d_3d_4^2 - d_3^2 - d_4^3 + 4$,
    \item from \textbf{Case 6}, we have $\sigma(T)= d_1d_2^2 - d_2d_3^2 + d_3^2 - d_2^3 - 14$, 
    \item From \textbf{Case 7}, we have $\sigma(T)= -d_1^2 + 2d_1 - d_3d_4^2 + d_4^3 - 1$,
    \item From \textbf{Case 8,9,10} we have $\sigma(T)=-2d_1 + d_3d_4^2 - d_1d_2^2 + d_2^3 - d_4^3 + 14$.
\end{itemize}
 
Notice that from all cases,there are many objectionable relationships because they do not give the required results according to the terms we have established, so they were excluded and the cases that meet these terms were discussed, through which we obtain, so that we have: 
\[
\sigma(T)=\begin{cases}
    \sigma_{\max}(T)= d_1^3 + 2d_1^2 + d_1 + d_2d_3^2 - d_1d_2^2 + d_4^3 - d_3d_4^2 - d_2^3  \\
    \sigma_{\min}(T)= -d_1^2 + 2d_1 - d_3d_4^2 + d_4^3 - 1.
\end{cases}
\]
As desire.
\end{proof}

We want to refer to Hypothesize~\ref{hy.sigma5} according to the constant term $(d_1-1)^3+(d_4)^3$ that is employed in Hypothesize~\ref{hy.sigma6} strongly influence outcomes for  amplifying the term $(d_1-1)^3+(d_6-1)^3$.
\begin{hypothesize}\label{hy.sigma6}
let $\mathscr{D}=(d_1,\dots,d_5)$  be a degree sequence where $d_6\geqslant d_5 \geqslant d_4 \geqslant d_3 \geqslant d_2 \geqslant d_1$, then Sigma index of $T$ is: 
\[
\sigma(T)=\sum_{i=1}^{n-2}d_i(d_{i+1})^2+(d_1-1)^3+(d_6-1)^3.
\]
where inequality holds if and only if $d_i=d_{i-1}+1$.
\end{hypothesize}
\begin{proof}
We will discuss two basic conditions in the proof, one of which leads to the other, as it is fulfilled under certain conditions according to Hypothesize~\ref{hy.sigma5} and Lemma~\ref{le.sigma2} as follows: 
\case{a} In this case, we assume the sequence is $\mathscr{D}=(d_1,d_2,d_3,d_4,d_5,d_6)$ where $d_6\geqslant d_5 \geqslant d_4 \geqslant d_3 \geqslant d_2 \geqslant d_1$, in this case the term is $d_i=d_{i-1}+1$, the number of possible cases in this case is 720 cases, some of which give similar results and some of which do not fulfil the condition and give negative values, so we took 10 cases and compared them with the rest of the cases and reviewed the results of this comparison through cases from case 1 to case 10.
\case{b} In this case, we assume the sequence is $\mathscr{D}=(d_1,d_2,d_3,d_4,d_5,d_6)$ where $d_6\geqslant d_5 \geqslant d_4 \geqslant d_3 \geqslant d_2 \geqslant d_1$, this case is a generalisation of the previous case but with no specific conditions, just a normal order.

Therefore, let begin with first \textbf{Case a} as following: 
\case{1} In this case, we have
\[
 \sigma(T_1)=\begin{cases}
     2d_1^3 - 4d_1^2 + 3d_1 + d_1d_2^2 + d_2d_3^2 + d_3d_4^2 + d_5^3+4 & d=(d_1,d_2,d_3,d_4,d_5,d_6), \\
     2d_1^3 + (d_2)(d_3)^2 + 2d_3(d_4)^2 + d_1(d_2)^2 - 3d_1^2 + 3d_1 + 33 & d=(d_2,d_4,d_5,d_3,d_6,d_1), \\
     2d_1^3 - 5d_1^2 + 5d_1 + d_2d_3^2 + d_1d_2^2 + d_3^3 + d_5^3 + 24 & d=(d_4,d_5,d_3,d_1,d_2,d_6), \\
     2d_1^3 - 5d_1^2 + 5d_1 + d_2^3 + d_3^3 + d_3d_4^2 + d_4d_5^2 + 54 & d=(d_3,d_1,d_6,d_2,d_5,d_4), \\
     2d_1^3 - 5d_1^2 + 5d_1 + d_2^3 + d_2d_3^2 + d_3d_4^2 + d_4^3 + 17 & d=(d_3,d_1,d_2,d_5,d_4,d_6).
 \end{cases}
\]

\case{2} In this case, we have the constant term is differ $d_1$, by considering $\beta=d_2d_3^2 + d_3^3 + 2d_3^2 + d_3 + d_5^3$, $\gamma=d_4d_5^2 + 2d_4^2 + d_2d_4^2$ and $\zeta= d_1d_2^2 + d_3d_5^2 + d_5^2$ then we have: 
\[
\sigma(T_2)=\begin{cases}
    3d_1^3 - 3d_1^2 + 6d_1 + d_3d_4^2 + d_3^3 + d_5^3 + 17 & d=(d_6,d_5,d_3,d_2,d_1,d_4), \\
    (d_2)^3 + (d_1-1)d_1^2 + (d_1-2)(d_1-1)^2 + \beta + 13 & d=(d_3,d_2,d_1,d_4,d_5,d_6), \\
  d_5^3 + 3d_1^3 - 2d_1^2 + 4d_1 + d_1d_3^2 + d_3^2 + d_4^3 + 19  & d=(d_6,d_3,d_2,d_1,d_4,d_5), \\
  2d_1^3 - 4d_1^2 + 3d_1 + 2d_2^3 + 2d_2^2 + d_2 + \gamma + 49 & d=(d_1,d_6,d_2,d_4,d_5,d_3), \\ 
  2d_1^3 - 3d_1^2 + 3d_1 + 2d_3^3 + 2d_3^2 + d_3 +\zeta + 26 & d=(d_2,d_1,d_5,d_6,d_3,d_4).
\end{cases}
\]

Through these discussed cases, we can observe that both elements at the vertices of the tree have a constant coefficient of $(d_i)^3$, and we also observe that the elements located in the centre have a coefficient of the form $d_i(d_{i+1})^2$, so the difference remains in the value of the constant added to these elements, which varies from one case to another, according to \textbf{Case 1} we have: 
\[
\sigma(T_3)=\begin{cases}
    2d_1^3 - 3d_1^2 + 2d_1 + d_1d_2^2 + d_3d_4^2 - d_3^3 - 2d_3^2 - d_3 - 9, \\
    d_3^3 - d_2d_3^2 + d_4d_5^2 - d_4^3 + 37.
\end{cases}
\]
According to Case 2 we have: 
\[
\sigma(T_4)=\begin{cases}
    3d_1^3 - 2d_1^2 + 2d_1 + d_3d_4^2 - d_2^3 - d_2d_3^2 - 2d_3^2 - d_3 + 4, \\
    -d_1^3 - 4d_1^2 + 3d_1 + d_2^3 + d_2d_3^2 - d_3^3 - d_1d_2^2 - d_3d_5^2 - d_5^2 + d_5^3 - 13.
\end{cases}
\]
so that, notice terms dominate due to rapid growth according to $d_6\geqslant d_5 \geqslant d_4 \geqslant d_3 \geqslant d_2 \geqslant d_1$, and constants value, then we have:
\[
\sigma(T)=\begin{cases}
    \sigma_{\max}(T)=2d_1^3 - 4d_1^2 + 3d_1 + 2d_2^3 + 2d_2^2 + d_2 + \gamma + 49,  \\
    \sigma_{\min}(T)= (d_2)^3 + (d_1-1)d_1^2 + (d_1-2)(d_1-1)^2 + \beta + 13.
\end{cases}
\]
Now, let be considered the sequence $\mathscr{D}=(d_1,d_2,d_3,d_4,d_5,d_6)$, then we have according \textbf{Case b} where considering $k>0, k\in \mathbb{N}$,and the number of pendent vertices $n\geq 1$ then:  
\case{1} In this case, we have the sequence $\mathscr{D}=(d_1,d_2,d_3,d_4,d_5,d_6)$, then we have: 
\begin{align*}
    \sigma(T)&=(n-1)((d_6-1)^2+(d_1-1)^2)+(n-2)((d_2-1)^2+\\
    &+(d_3-1)^2+(d_4-1)^2+(d_5-1)^2)+k.\\
    &=(n-1)((d_1-1)^2+(d_6-1)^2)+(n-2)(\sum_{i=2}^{5}(d_i-1)^2)+k.
\end{align*}
\case{2} In this case, we have the sequence $\mathscr{D}=(d_6,d_4,d_2,d_1,d_3,d_5)$,  then: 
\[
    \sigma(T)=(n-1)((d_6-1)^2+(d_5-1)^2)+(n-2)((d_4-1)^2+(d_2-1)^2+(d_1-1)^2+(d_3-1)^2)+k.
\]
Therefore, we have: 
\[
    \sigma(T)=(n-1)(\sum_{i=5}^{6}(d_i-1)^2)+(n-2)(\sum_{i=2}^{4}(d_i-1)^2)+k.
\]
Note that in both cases we need to know the bifurcations generated by each of the vertices, which we were able to do in \textbf{Case a} discussed above, so we can formulate anchors the maximum, while the smallest $d_6$ is absent in most expressions, reducing exact numerical bounds, specific values impact show in the following: 
\[
\sigma(T)=\begin{cases}
    \sigma_{\max}(T)=(n-1)(\sum_{i=5}^{6}(d_i-1)^2)+(n-2)(\sum_{i=2}^{4}(d_i-1)^2)+k \\
    \sigma_{\min}(T)= (n-1)((d_1-1)^2+(d_6-1)^2)+(n-2)(\sum_{i=2}^{5}(d_i-1)^2)+k.
\end{cases}
\]
As desire.
\end{proof}
\begin{theorem}~\label{thm.sigma.hy}
Let $T$ be a tree of order $n$, let $\mathscr{D}=(d_1,\dots,d_n)$  be a degree sequence where $d_1\geqslant d_2\geqslant \dots \geqslant d_n$, then Albertson index among tree $T$ is: 
\[
\sigma(T)=(d_n-1)^3+(d_1-1)^3+\sum_{i=1}^{n}(d_i-d_{i+1})^2+\sum_{i=2}^{n-1}(d_i-1)^2(d_i-2).
\]
\end{theorem}

Note the location of the vertices in the sequence relative to the maximum value and the minimum value as a result of considering all the possibilities for this relationship.

\section{Conclusion}\label{sec5}
Through this paper, we presented a study of the topological indices defined by the Sigma Index and the Albertson Index according to a degree sequence $d_n\geqslant \dots \geqslant d_1$, in Hypothesize~\ref{hy.sigma6} for a degree sequence $\mathscr{D}=(d_1,\dots,d_6)$ where $d_6\geqslant d_5 \geqslant \dots \geqslant d_1$, then we defined Sigma index as $\sigma(T)=\sum_{i=1}^{n-2}d_i(d_{i+1})^2+(d_1-1)^3+(d_6-1)^3$.
The examination of topological indices on trees gives us a general overview through bounds to find the maximum and minimum bounds and this is what we were keen to show in the proofs we presented for each case discussed.

\end{document}